\documentclass[a4paper,12pt]{amsart}
\usepackage{amsmath,amssymb,amsxtra,comment,graphicx,psfrag}
\usepackage{bm,mathrsfs}
\usepackage{mathtools}
\usepackage{xspace}
\usepackage{color}
\usepackage{xcolor}
\usepackage{hhline}

\definecolor{forestgreen}{rgb}{0.13, 0.55, 0.13}
\definecolor{lightblue}{rgb}{0.68, 0.85, 0.9}
\usepackage[colorlinks=true,
            linkcolor=blue,
            urlcolor  = blue,
            citecolor=forestgreen,
            anchorcolor = lightblue]{hyperref}

\usepackage{enumerate}

\usepackage{hyperref}
\usepackage{pdfsync}
\usepackage{hhline}
\usepackage{fancyhdr}
\usepackage{epsfig}
\usepackage{epsfig,subfigure,epstopdf}
\allowdisplaybreaks

\usepackage{pstricks,pst-plot,pst-func}
\usepackage{pspicture}
\usepackage{curves}

\include{rgb}

\def\d{{\rm d}}

\def\N{{\mathbb N}}

\definecolor{mygray}{rgb}{0.9,0.9,0.9}

\begin{document}
\theoremstyle{plain}
\newtheorem{theorem}{Theorem}[section]
\newtheorem{lemma}{Lemma}[section]
\newtheorem{proposition}{Proposition}[section]
\newtheorem{corollary}{Corollary}[section]

\theoremstyle{definition}
\newtheorem{definition}[corollary]{Definition}

\newtheorem{example}{Example}[section]

\newtheorem{remark}{Remark}[section]
\newtheorem{remarks}[remark]{Remarks}
\newtheorem{note}{Note}
\newtheorem{case}{Case}

\numberwithin{equation}{section}
\numberwithin{table}{section}
\numberwithin{figure}{section}

\title[Weighted and shifted BDF2 methods on variable grids ]
{Weighted and shifted BDF2 methods on variable grids}

\author[Minghua Chen]{Minghua Chen}
\address{School of Mathematics and Statistics, Gansu Key Laboratory of Applied Mathematics and Complex Systems,
 Lanzhou University, Lanzhou 730000, P.R. China}
\email {\href{mailto:chenmh@lzu.edu.cn}{chenmh{\it @\,}lzu.edu.cn}}

%

\author[Fan Yu]{Fan Yu}
\address{School of Mathematics and Statistics, Gansu Key Laboratory of Applied Mathematics and Complex Systems,
 Lanzhou University, Lanzhou 730000, P.R. China}
\email {\href{mailto:yuf20@lzu.edu.cn}{yuf20{\it @\,}lzu.edu.cn}}

\author[Qingdong Zhang]{Qingdong Zhang}
\address{School of Mathematics and Statistics, Gansu Key Laboratory of Applied Mathematics and Complex Systems,
 Lanzhou University, Lanzhou 730000, P.R. China}
\email {\href{mailto:zhangqd19@lzu.edu.cn}{zhangqd19{\it @\,}lzu.edu.cn}}





\begin{abstract}
Variable steps implicit-explicit multistep methods  for  PDEs have been presented in \cite{WR:08}, where  the zero-stability  is studied for ODEs;
however, the stability analysis  still remains an open question for PDEs.
Based on the idea of linear multistep methods, we   present a simple weighted and shifted BDF2  methods with variable steps for the parabolic problems,
which serve as a bridge between BDF2 and Crank-Nicolson scheme.
The contributions of this paper are as follows:
we first prove that the optimal  adjacent  time-step  ratios  for the weighted and shifted BDF2, which greatly improve the maximum time-step ratios for BDF2 in \cite{LZ,TZZ:20}.
Moreover, the unconditional stability and optimal  convergence are rigorous proved, which make up for the vacancy of the  theory for PDEs in \cite{WR:08}. Finally, numerical experiments are given to illustrate theoretical results.

\end{abstract}

\maketitle


\section{Introduction}\label{Se:intro}
Let $T >0, u^0\in H,$ and consider the initial value
problem of seeking $u \in C((0,T];D(A))\cap C([0,T];H)$ satisfying
\begin{equation}\label{ivp}
\left \{
\begin{aligned}
&u' (t) + Au(t) = f(t), \quad 0<t<T,\\
& u(0)=u^0,
\end{aligned}
\right .
\end{equation}
with  $A$ a positive definite, selfadjoint, linear operator on a
Hilbert space $(H, (\cdot , \cdot )) $ with domain  $D(A)$
dense in $H$ and $f  : [0,T] \to H$ a given forcing term.

Let $N\in \N$ and choose the nonuniform time levels $0=t_0<t_1<\cdots<t_N=T$ with the time-step $\tau_k=t_k-t_{k-1}$ for $1\leq k\leq N$. For any time sequence $\{v^n\}_{n=0}^N$, denote
\begin{equation*}
\nabla_{\tau}v^{n}:=v^n-v^{n-1},\quad\partial_{\tau}v^n:=\nabla_{\tau}v^{n}/\tau_n,\quad v^{n-\frac{1}{2}}:=\frac{1}{2}\left(v^n+v^{n-1}\right).
\end{equation*}
For $k=1,2,$ let $\Pi_{n,k}v$ denote the interpolating polynomial of a function $v$ over $k+1$ nodes $t_{n-k},\cdots,t_{n-1}$ and $t_n$. Taking $v^n=v(t_n)$, by using the Lagrange interpolation, the BDF1 formula is defined by
$D_1v^n:=(\Pi_{n,1}v)'(t)=\nabla_{\tau}v^{n}/\tau_n$ for $n\geq 1$, and the BDF2 formula is defined by
\begin{equation*}
\widetilde{D}_2v^n=\left(\Pi_{n,2}v\right)'(t_n)=\frac{1+2r_n}{\tau_n(1+r_n)}\nabla_{\tau}v^{n}-\frac{r_n^2}{\tau_n(1+r_n)}\nabla_{\tau}v^{n-1}, ~~\forall ~n\geq2,
\end{equation*}
where the adjacent time step ratios
\begin{equation*}
r_k:=\frac{\tau_k}{\tau_{k-1}},~~\forall~ 2\leq k\leq N.
\end{equation*}
Similarly, we construct the shifted BDF2 formula
\begin{equation*}
\bar{D}_2v^n=\left(\Pi_{n,2}v\right)'(t_{n-1})=\frac{1}{\tau_n(1+r_n)}\nabla_{\tau}v^{n}+\frac{r_n^2}{\tau_n(1+r_n)}\nabla_{\tau}v^{n-1}, ~~\forall ~n\geq2.
\end{equation*}
Hence, the weighted and shifted BDF2 (WSBDF2) operator is defined by
$$D_2v^n:=\theta \widetilde{D}_2v^n+(1-\theta)\bar{D}_2v^n,$$
i.e.,
\begin{equation}\label{1.2}
D_2v^n=\frac{1+2\theta r_n}{\tau_n(1+r_n)}\nabla_{\tau}v^{n}+\frac{(1-2\theta) r_n^2}{\tau_n(1+r_n)}\nabla_{\tau}v^{n-1},~~\forall \theta\in \left[1/2,1\right].
\end{equation}
Note that  $\theta\in \left[1/2,1\right]$ would  be a suitable choice, since  the scheme  is not  $A$-stable if $\theta<1/2$ and
the maximum ratios is too narrow if $\theta>1$, see remark \ref{remark1.1} and remark \ref{remark2.1}, respectively.
We remark that the coefficients of  WSBDF2 operator \eqref{1.2} share similarities with the implicit-explicit multistep methods in  \cite{WR:08}.
However, the techniques to acquire the coefficients still have some significant differences  between these two schemes.
For example,   the implicit-explicit multistep methods are obtained by the order conditions, but  WSBDF2 methods are constructed  by the weighted and shifted technique, which is
more simple and efficient for designing  the high-order schemes.

For concreteness, we use the BDF1 scheme, by defining $D_2v^1:=D_1v^1$ to compute the first-level
solution $u^1$ because the two-step WSBDF2 method needs two starting values. We recursively define a sequence of approximations $u^n$ to
the nodal values $u(t^n)$ of the exact solution by the WSBDF2 method,
\begin{equation}\label{1.3}
D_2u^n-\theta\Delta u^n-(1-\theta)\Delta u^{n-1}=\theta f^n+(1-\theta)f^{n-1},
\end{equation}
where the initial data $u^0=u_0$ and the exterior force $f^n=f\left(t_n\right)$. The weak form of the
time-discrete problem \eqref{1.3} for $k\geq1$ reads
\begin{equation}\label{1.7}
(D_2u^k,\phi^k)-(\theta\Delta u^k+(1-\theta)\Delta u^{k-1},\phi^k)=(\theta f^k+(1-\theta)f^{k-1},\phi^k)~~~~\forall\phi^k\in H_0^1(\Omega),
\end{equation}
where $(\cdot,\cdot)$ denotes the usual inner product in the space $L^2(\Omega)$. Correspondingly, $\|\cdot\|$ denotes the associated $L^2$ norm. There exists a positive constant $C_\Omega$
dependent on the domain $\Omega$ such that $\|\phi^n\|\leq C_\Omega\|\nabla\phi^n\|$ for any $\phi^n\in L^2(\Omega)\cap H_0^1(\Omega)$.
In particular, \eqref{1.3} reduces to the BDF2 method (multistep method) when $\theta=1$
\begin{equation*}
\widetilde{D}_2u^n-\Delta u^n=f^n,
\end{equation*}
and the Crank-Nicolson  method (one-step method) when $\theta=1/2$
\begin{equation*}
\frac{u^{n}-u^{n-1}}{\tau_n}-\Delta u^{n-\frac{1}{2}}=f^{n-\frac{1}{2}},
\end{equation*}
where $f^{n-\frac{1}{2}}=f(t_{n-\frac{1}{2}})$ or $f^{n-\frac{1}{2}}=\frac{1}{2}\left(f(t_{n-1})+f(t_n)\right)$ are permissible, as one sees using the Peano kernel for the midpoint or trapezoidal rule.

\begin{remark}\label{remark1.1}
The initial value problem  of \eqref{ivp} takes the form $u' (t)=f(u,t)$,
then \eqref{1.3} reduced  to the following simple form
\begin{gather*}
\left( \theta+\frac{1}{2} \right)u^{n+2}-2\theta u^{n+1}+\left(\theta-\frac{1}{2} \right)u^{n}= \tau \left( \theta f^{n+2}+  \left(1-\theta\right) f^{n+1}\right), \tag{$*$}
\end{gather*}
where $\tau$ is the  uniform time stepsize.

In fact, the above scheme ($*$) is a special type of two-step linear multistep methods, which  has the   characteristic polynomials
\begin{equation*}
\rho(\xi) = \left( \theta+\frac{1}{2} \right) \xi^{2}-2\theta \xi+  \left(\theta-\frac{1}{2} \right)~~{\rm and}~~\sigma(\xi) = \theta \xi^{2}+  \left(1-\theta\right)\xi.
\end{equation*}
From Lemma 2.1 in \cite{Dahl}, we know that the two-step method is $A$-stable, which required the roots of the second  characteristic polynomial $\sigma(\xi)$
lie on or within the unit circle. Namely,
\begin{equation*}
\xi =\left|\frac{1-\theta }{\theta}\right|\leq 1~~{\rm if ~ and ~only ~if~} \theta\geq \frac{1}{2}.
\end{equation*}

\end{remark}

The WSBDF2 operator \eqref{1.2} is regarded as a discrete convolution summation
\begin{equation}\label{1.4}
D_2v^n=\sum_{k=1}^{n}b_{n-k}^{(n)}\nabla_\tau v^k,~~\forall n\geq1,
\end{equation}
where the discrete convolution kernels $b_{n-k}^{(n)}$ are defined by $b_0^{(1)}=1/\tau_1$, and when $n\geq2$,
\begin{equation}\label{1.5}
b_{0}^{(n)}:=\frac{1+2\theta r_n}{\tau_n(1+r_n)},~~~b_{1}^{(n)}:=\frac{(1-2\theta) r_n^2}{\tau_n(1+r_n)} ~~~~{\rm and}~~~~ b_{j}^{(n)}:=0 ~~~~{\rm for}~~~~2\leq j\leq n-1.
\end{equation}

Following the approach of \cite{LZ}, the discrete orthogonal convolution (DOC) kernels $\{d_{n-k}^{(n)}\}_{k=1}^n$ is defined by
\begin{equation}\label{1.6}
d_{0}^{(n)}:=\frac{1}{b_{0}^{(n)}}\quad{\rm and}\quad d_{n-k}^{(n)}:=-\frac{1}{b_0^{(k)}}\sum_{j=k+1}^nd_{n-j}^{(n)}b_{j-k}^{(j)}=-\frac{b_{1}^{(k+1)}}{b_{0}^{(k)}}d_{n-k-1}^{(n)} ~~~~{\rm for}~~~~1\leq k\leq n-1.
\end{equation}
Obviously, the DOC kernels $\{d_{n-k}^{(n)}\}_{k=1}^n$ satisfies the discrete orthogonal identity
\begin{equation}\label{1.8}
\sum_{j=k}^n d_{n-j}^{(n)}b_{j-k}^{(n)}\equiv\delta_{nk}  ~~~~{\rm for}~~~~1\leq k\leq n.
\end{equation}
It is to note that the positive semi-definiteness of WSBDF2 convolution kernels $b_{n-k}^{(n)}$ and the corresponding DOC kernels $d_{n-k}^{(n)}$ plays an important role in our numerical analysis.

For convenience, we introduce the following matrices:
\begin{equation}\label{1.9}
B:=\begin{pmatrix}
b_0^{(1)}&&&\\b_1^{(2)}&b_0^{(2)}&&\\&\ddots&\ddots&\\&&b_1^{(n)}&b_0^{(n)}
\end{pmatrix}\quad{\rm and}\quad
D:=\begin{pmatrix}
d_0^{(1)}&&&\\d_1^{(2)}&d_0^{(2)}&&\\ \vdots&\vdots&\ddots&\\d_{n-1}^{(n)}&d_{n-2}^{(n)}&\cdots&d_0^{(n)}
\end{pmatrix},
\end{equation}
where the discrete convolution kernels $b_{n-k}^{(n)}$ and the DOC kernels $\{d_{n-k}^{(n)}\}_{k=1}^n$ are defined in \eqref{1.5} and \eqref{1.6}, respectively. It follows from the
discrete orthogonal identity \eqref{1.8} that:
\begin{equation}\label{1.10}
D=B^{-1}.
\end{equation}

Variable steps implicit-explicit multistep methods for PDEs have been presented in \cite{ARW,WR:08}, where  the zero-stability is studied for ODEs;
however, the stability and convergence of multi-step time-stepping schemes with variable steps would
be challenging difficult for PDEs, which is the motivation for us to consider this paper. Nowadays, there are many researchers study on the numerical analysis with variable steps for parabolic problems.
Variable steps BDF2 method for ODEs, Grigorieff proved that it is zero-stable if the adjacent time-step ratios
$r_k<1+\sqrt{2}$ in \cite{G}, also see \cite{CL}. In \cite{Be}, Becker applied the variable-step BDF2 formula to the parabolic equation and
established a second-order temporal convergence if $r_k\leq(2+\sqrt{13})/3\approx1.868$. However, the
resulting error estimate is far from sharp because it involves an undesired prefactor $\exp(C\Gamma_n)$, where $\Gamma_n$
may be unbounded as the time-step sizes vanish \cite {T}. Emmrich \cite{E} improved $r_k\leq1.91$ but still retained the undesirable prefactor $\exp(C\Gamma_n)$. Furthermore, \cite{CWYZ} replaced  $\exp(C\Gamma_n)$ by a bounded exponential prefactor $\exp(Ct_n)$ if $r_k\leq1.53$.
Recently, an adaptive BDF2 scheme for linear diffusion equation
is considered under $r_k\leq(3+\sqrt{17})/2\approx3.561$ in \cite{LZ}
and extended to $r_k\approx4.8645$ in \cite{LJWZ,TZZ:20}.
Using Crank-Nicolson reconstructions technique, a posteriori error estimate for parabolic equations with variable steps has been obtained \cite{AMN}.
In this work,  we   present the simple WSBDF2 methods with variable steps for the parabolic problems,
which serve as a bridge between BDF2 and Crank-Nicolson scheme.
We  prove that the optimal  adjacent  time-step  ratios  for the WSBDF2 schemes, which greatly improve the maximum time-step ratios for BDF2 in \cite{LZ,LJWZ,TZZ:20}.
Based on DOC technique \cite{LZ}, the unconditional stability and optimal convergence are rigorous proved, which fill in the gap  of the   theory for PDEs in \cite{WR:08}.

An outline of the paper is as follows. In the next section, the upper bound for the optimal  adjacent time-step ratios $r_k$ is proved so that the discrete convolution kernels $b_{n-k}^{(n)}$ are positive semi-definite, which will play an core role in the stability and convergence analysis.
In Section \ref{Se:stab}, the energy stability and unconditional stability are proved for WSBDF2  schemes under certain restrictions on the adjacent time-step ratios.
The  optimal convergence is rigorous proved for WSBDF2 methods in Section \ref{Se:conv}.
In Section \ref{Se:numer}, numerical examples are implemented to validate the theoretical results.
Finally, we conclude the paper with some remarks.

\section{Upper bound for the adjacent time-step ratios}\label{Se:upper}
We prove the upper bound for the adjacent time-step ratios so that the discrete
convolution kernels $b_{n-k}^{(n)}$ are positive semi-definite, which will be essential to the stability and
convergence analysis.
\begin{lemma}\cite[p.\,28]{Quarteroni:07}\label{adlemm2.1}
A real matrix $L$ of order $n$ is positive definite  if and only if  its symmetric part $H=\frac{L+L^T}{2}$ is positive definite.
\end{lemma}
\begin{lemma}\label{Le:upper}
Let the discrete convolution kernels $b^{(n)}_{n-k}$ be defined in \eqref{1.5}. Then for any real sequence $\{w_k\}^n_{k=1}$, it holds that
\begin{equation*}
\sum^n_{k=1}w_k\sum^k_{j=1}b^{(k)}_{k-j}w_j\geq0 ~~~~{\rm for}~~n \geq 1.
\end{equation*}
\end{lemma}
\begin{proof}
It is obvious that
\begin{equation*}
\sum^n_{k=1}w_k\sum^k_{j=1}b^{(k)}_{k-j}w_j=W^TBW ~~~~{\rm for}~~n \geq 1,
\end{equation*}
where $W=(w_1,w_2,\cdots,w_n)^T$ and the matrix $B$ is defined in \eqref{1.9}.

We introduce the symmetric tridiagonal matrix $\widetilde{B}:=B+B^T=(\widetilde{b}_{ij})$  with entries
\begin{equation*}
\widetilde{b}_{i,j}=\left \{
\begin{aligned}
&2b_0^{(i)},\quad 1\leq i=j\leq n,\\
&b_1^{(i)}, \quad j=i-1,\quad i=2,\dotsc,n,\\
&b_1^{(j)}, \quad j=i+1,\quad i=1,\dotsc,n-1,
\end{aligned}
\right .
\end{equation*}
and all other entries equal zero.

We only need to prove that the matrix $\widetilde{B}$ is positive definite.

By linear transformation, the matrix $\widetilde{B}$ can be transformed into an upper triangular matrix $L=(\ell_{ij})$  with entries
\begin{equation}\label{2.1}	
\ell_{i,j}=\left \{
\begin{aligned}
&2b_0^{(1)}, \quad i=j=1,\\
&2b_0^{(i)}-\frac{1}{\ell_{i-1,i-1}}\left(b_1^{(i)}\right)^2,\quad 2\leq i=j\leq n,\\
&b_1^{(j)}, \quad j=i+1,\quad i=1,\dotsc,n-1,
\end{aligned}
\right .
\end{equation}
and all other entries equal zero.
With this notation, we have $\ell_{1,1}=2b_0^{(1)}=\frac{2}{\tau_1}$.
The main task is to prove
\begin{equation}\label{2.2}	
\ell_{i,i}\geq \frac{2}{\tau_i}, \quad 2\leq i\leq n.
\end{equation}

Below we use the mathematical induction to  prove \eqref{2.2}.
For $i=2,$ we have
\begin{equation*}
\begin{split}
\ell_{2,2}&=2b_0^{(2)}-\frac{1}{\ell_{1,1}}\left(b_1^{(2)}\right)^2=\frac{2\left(1+2\theta r_2\right)}{\tau_2\left(1+r_2\right)}-\frac{\tau_1\left(1-2\theta\right)^2r_2^4}{2\tau_2^2\left(1+r_2\right)^2}\\
&=\frac{2\left(1+2\theta r_2\right)}{\tau_2(1+r_2)}-\frac{(1-2\theta)^2r_2^3}{2\tau_2(1+r_2)^2}
=\frac{4\left(1+2\theta r_2\right)(1+r_2)-(1-2\theta)^2r_2^3}{2\tau_2(1+r_2)^2}\\
&=\frac{4(1+r_2)^2-4(1+r_2)^2+4\left(1+2\theta r_2\right)(1+r_2)-(1-2\theta)^2r_2^3}{2\tau_2(1+r_2)^2}\\
&=\frac{2}{\tau_2}+\frac{-4\left(1+r_2\right)^2+4\left(1+2\theta r_2\right)\left(1+r_2\right)-\left(1-2\theta\right)^2r_2^3}{2\tau_2(1+r_2)^2}.
\end{split}
\end{equation*}
If the time-step ratios $0<r_k\leq r_p$, where the suboptimal  adjacent  time-step  ratios
\begin{equation*}
r_p=\frac{2+2\sqrt{2\theta}}{2\theta-1},
\end{equation*}
is the positive root of the equation $\left(1-2\theta\right)r_p^2+4r_p+4=0$, then the following inequality holds
\begin{equation*}
\begin{split}
&-4\left(1+r_2\right)^2+4\left(1+2\theta r_2\right)\left(1+r_2\right)-\left(1-2\theta\right)^2r_2^3\\
&=\left(2\theta-1\right)r_2\left(\left(1-2\theta\right)r_2^2+4r_2+4\right)\geq0.
\end{split}
\end{equation*}
Therefore, E.q. \eqref{2.2} is proven for the case $i=2.$

Suppose now \eqref{2.2} holds for all $2\leq i\leq n-1$, we need to prove that it holds for $i=n.$
For $i=n,$ using the induction assumption, we get
\begin{equation*}
\begin{split}
\ell_{n,n}&=2b_0^{(n)}-\frac{1}{\ell_{n-1,n-1}}\left(b_1^{(n)}\right)^2=\frac{2\left(1+2\theta r_n\right)}{\tau_n\left(1+r_n\right)}-\frac{\tau_{n-1}\left(1-2\theta\right)^2r_n^4}{2\tau_n^2\left(1+r_n\right)^2}\\
&=\frac{2\left(1+2\theta r_n\right)}{\tau_n(1+r_n)}-\frac{(1-2\theta)^2r_n^3}{2\tau_n(1+r_n)^2}
=\frac{4\left(1+2\theta r_n\right)(1+r_n)-(1-2\theta)^2r_n^3}{2\tau_n(1+r_n)^2}\\
&=\frac{4(1+r_n)^2-4(1+r_n)^2+4\left(1+2\theta r_n\right)(1+r_n)-(1-2\theta)^2r_n^3}{2\tau_n(1+r_n)^2}\\
&=\frac{2}{\tau_n}+\frac{-4\left(1+r_n\right)^2+4\left(1+2\theta r_n\right)\left(1+r_n\right)-\left(1-2\theta\right)^2r_n^3}{2\tau_n(1+r_n)^2}.
\end{split}
\end{equation*}
By the similar way, we have
\begin{equation*}
-4\left(1+r_n\right)^2+4\left(1+2\theta r_n\right)\left(1+r_n\right)-\left(1-2\theta\right)^2r_n^3\geq0,~~0<r_n\leq r_p.
\end{equation*}
Therefore, E.q. \eqref{2.2} is proven for the case $i=n.$

Hence, the matrix $\widetilde{B}$ is positive definite by Sylvester's criterion \cite[p.\,439]{HJ},
since every leading or trailing principal minor of $\widetilde{B}$ is positive, i.e., $\det |\widetilde{B}_{l\times l}|=\Pi_{i=1}^l l_{i,i}>0,~~\forall l\geq1$.
Consequently, the matrix $B$ is also positive definite   by Lemma \ref{adlemm2.1}.
\end{proof}

From the proof of  Lemma \ref{adlemm2.1}, we obtain the suboptimal and intuitive estimates for   time-step ratios
\begin{equation}\label{ad2.3}
r_p=\frac{2+2\sqrt{2\theta}}{2\theta-1}.
\end{equation}
We next give the optimal time-step ratios $r_s$ (see Figure \ref{Fig.ratio}), namely,
\begin{equation}\label{ad2.a4}
r_s=\frac{4\theta^2-\sqrt[3]{-4\theta^2E+3(1-2\theta)^2\frac{-F+\sqrt{G}}{2}}-\sqrt[3]{-4\theta^2E+3(1-2\theta)^2\frac{-F-\sqrt{G}}{2}}}{3(1-2\theta)^2}
\end{equation}
with
\begin{equation*}
\begin{split}
E&=16\theta^4+48\theta^3-48\theta^2+12\theta,\\
F&=16\theta^3+36\theta^2-36\theta+9,\\
G &=384\theta^5+912\theta^4-2496\theta^3+1944\theta^2-648\theta+81.
\end{split}
\end{equation*}
\begin{figure}[htb]
  \centering
  \includegraphics[width=6cm]{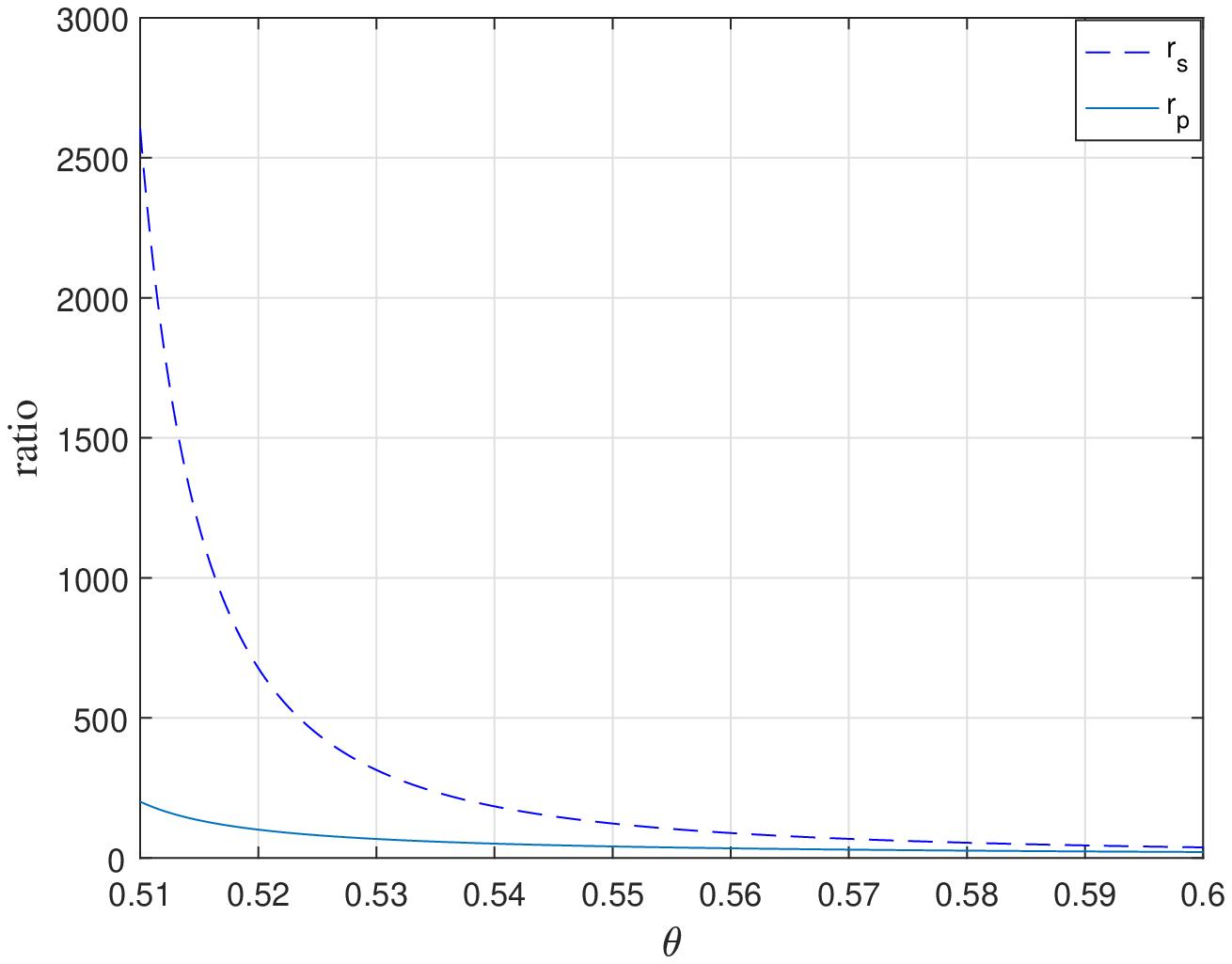}~~
    \includegraphics[width=6cm]{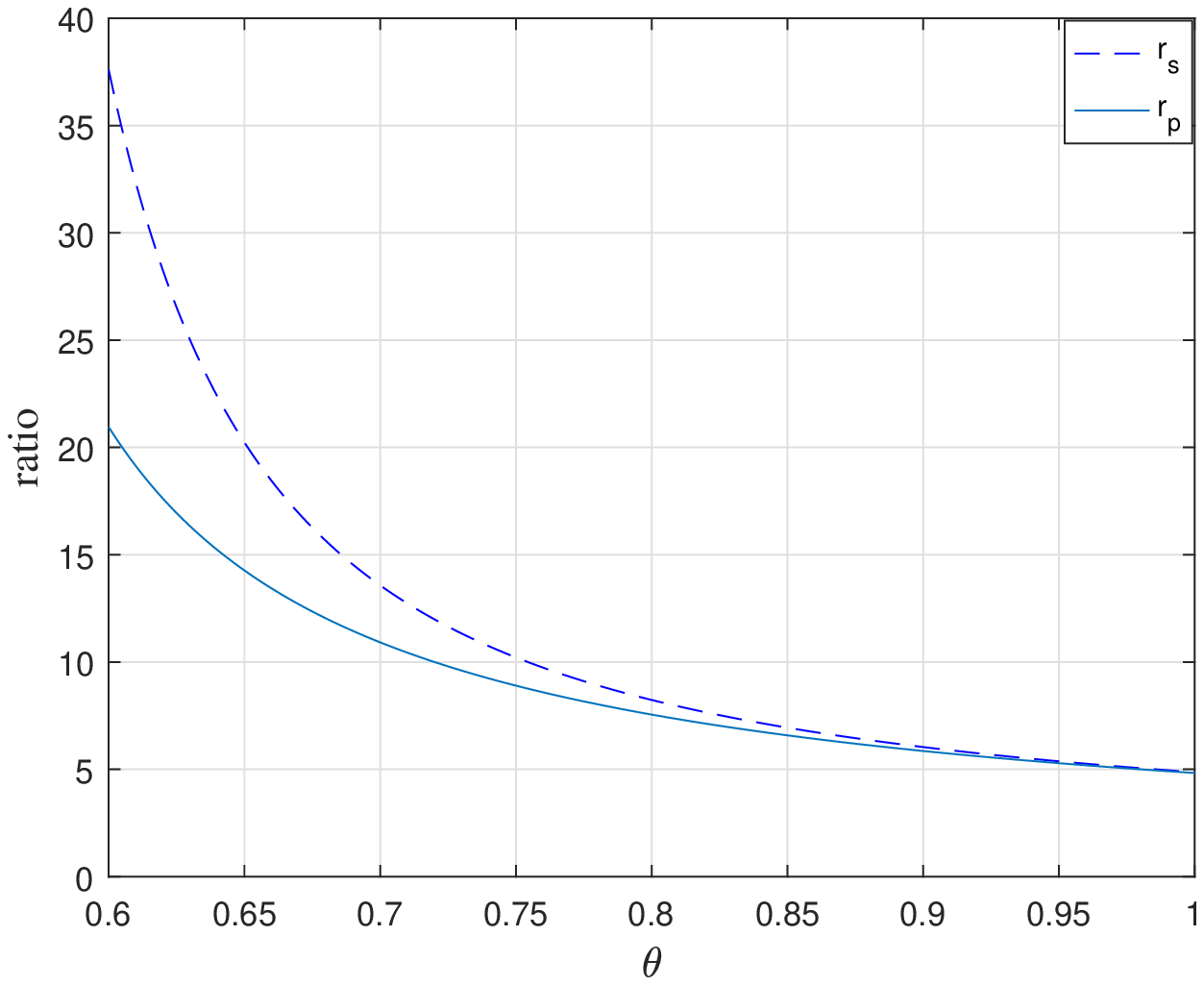}
    \caption{\small The graphs of the ratios $r_p$ and $r_s$ of  \eqref{ad2.3}  and \eqref{ad2.a4}.}  \label{Fig.ratio}
\end{figure}
\begin{lemma}\label{Le:WSBDF}
Let $r_s$ be given in \eqref{ad2.a4} and the adjacent step ratios $r_k$ satisfy $0<r_k\leq r_s$. For any real sequence $\{\omega_k\}_{k=1}^n$, it holds that
\begin{equation*}
2\omega_k\sum_{j=1}^kb_{k-j}^{(k)}\omega_j\geq\frac{(2\theta-1)r_{k+1}^{3/2}}{(1+r_{k+1})}\frac{\omega_{k}^2}{\tau_k}
-\frac{(2\theta-1)r_{k}^{3/2}}{(1+r_{k})}\frac{\omega_{k-1}^2}{\tau_{k-1}}\quad \forall k\geq2.
\end{equation*}
So the discrete convolution kernels $b_{n-k}^{(n)}$ defined in \eqref{1.5} are positive semi-definite,
\begin{equation*}
\sum_{k=1}^n\omega_k\sum_{j=1}^kb_{k-j}^{(k)}\omega_j\geq0.
\end{equation*}
\end{lemma}
\begin{proof}
Applying the inequality $-2ab\geq -a^2-b^2$ and taking  $u_k=\omega_k/\sqrt{\tau_k}$, one has
\begin{equation*}
\begin{split}
2\omega_k\sum_{j=1}^kb_{k-j}^{(k)}\omega_j&=2\tau_kb_{0}^{(k)}u_k^2+2\sqrt{\tau_k\tau_{k-1}}b_1^{(k)}u_ku_{k-1}\\
&\geq 2\tau_kb_{0}^{(k)}u_k^2+\sqrt{\tau_k\tau_{k-1}}b_1^{(k)}\left(u_k^2+u_{k-1}^2\right)\\
&=\frac{2\left(1+2\theta r_k\right)+\left(1-2\theta\right) r_k^{3/2}}{1+r_k}u_k^2-\frac{\left(2\theta-1\right)r_k^{3/2}}{ (1+r_k)}u^2_{k-1}\\
&=\frac{\left(2\theta-1\right)r_{k+1}^{3/2}}{ (1+r_{k+1})}\frac{\omega^2_{k}}{\tau_k}-\frac{\left(2\theta-1\right)r_k^{3/2}}{ (1+r_k)}\frac{\omega^2_{k-1}}{\tau_{k-1}}\\
&\quad+\left(\frac{2\left(1+2\theta r_k\right)+\left(1-2\theta\right) r_k^{3/2}}{1+r_k}-\frac{\left(2\theta-1\right)r_{k+1}^{3/2}}{ (1+r_{k+1})}\right)\frac{\omega_{k}^2}{\tau_{k}}\quad \forall k\geq2.
\end{split}
\end{equation*}
Let $r_s$ in \eqref{ad2.a4}
be  the positive root of the following  equation
$$(1-2\theta)^2r_s^3-4\theta^2r_s^2-4\theta r_s-1=0.$$
Let
\begin{equation*}
h(x,y):=\frac{2\left(1+2\theta x\right)+\left(1-2\theta\right)x^{3/2}}{1+x}-\frac{(2\theta-1)y^{3/2}}{1+y}.
\end{equation*}
It is easy to check that  $h(x,y)$ is increasing in $(0, 1)$ and decreasing in $(1,r_s)$ with respect to $x$;
and $h(x,y)$ is decreasing with respect to $y$. Then we have
\begin{equation*}
h(x,y)\geq \min\left\{h(0,r_s),h(r_s,r_s)\right\}=\frac{(1-2\theta)^2r_s^3-4\theta^2r_s^2-4\theta r_s-1}{1+r_s}=0~~{\rm for}~~0\leq x,y \leq r_s.
\end{equation*}
Thus it follows that
\begin{equation*}
2\omega_k\sum_{j=1}^kb_{k-j}^{(k)}\omega_j\geq\frac{(2\theta-1)r_{k+1}^{3/2}}{(1+r_{k+1})}\frac{\omega_{k}^2}{\tau_k}-\frac{(2\theta-1)r_{k}^{3/2}}{(1+r_{k})}\frac{\omega_{k-1}^2}{\tau_{k-1}}~~~~
\quad~~~~\forall 0<r_k\leq r_s.
\end{equation*}
Therefore, we obtain
\begin{equation*}
\begin{split}
2\sum_{k=1}^{n}\omega_k\sum_{j=1}^kb_{k-j}^{(k)}\omega_j&\geq\frac{2}{\tau_1}\omega_1^2+\frac{(2\theta-1)r_{n+1}^{3/2}}{(1+r_{n+1})}\frac{\omega_n^2}{\tau_n}-\frac{(2\theta-1)r_{2}^{3/2}}{(1+r_{2})}\frac{\omega_1^2}{\tau_1}\\
&=\frac{(2\theta-1)r_{n+1}^{3/2}}{(1+r_{n+1})}\frac{\omega_n^2}{\tau_n}+\frac{2+2r_2-(2\theta-1)r_2^{3/2}}{(1+r_{2})}\frac{\omega_1^2}{\tau_1}\geq0,
\end{split}
\end{equation*}
where we use
$$\frac{2+2r_2-(2\theta-1)r_2^{3/2}}{1+r_{2}}=\frac{1+2(1-\theta)r_2}{1+r_{2}}+\frac{1}{2}h(r_2,r_2)\geq 0.$$
In particular, it leads to   $r_s=4.8645365123$ if  $\theta=1$ and $r_s=\infty$ if  $\theta=1/2$.
\end{proof}

\begin{remark}\label{remark2.1}
 From \eqref{2.1}, we can check  that
$\ell_{1,1}=\frac{1}{\tau_1}\cdot l_1$ and $\ell_{k,k}=\frac{1}{\tau_k}\cdot l_k$
with
\begin{equation}\label{ad2.a5}
l_1=2,~~l_k=\frac{2}{\left(1+r_k\right)^2}\left[ \left(1+r_k\right)\left(1+2\theta r_k\right)-\frac{1}{2l_{k-1}}\left(2\theta-1\right)^2r_k^3   \right], ~~\forall k\geq 2.
\end{equation}
We know that the matrix $\widetilde{B}$ is positive definite if and only if $l_k>0$.
Figure \ref{Fig.2.1} shows that the optimal maximum ratios $r_s=4.8645$ for variable size BDF2, since $l_k<0$ when $r_s=4.8646$.
Moreover, choosing  $\theta=1$ is better than $\theta>1$, since the maximum ratio is too narrow for the latter.
On the other hand, $l_k\equiv 2$ for $\theta=1/2$, which implies that the ratio is $r_s=\infty$ (without restriction).
\end{remark}

\begin{figure}[htb]
  \centering
  \includegraphics[width=6cm]{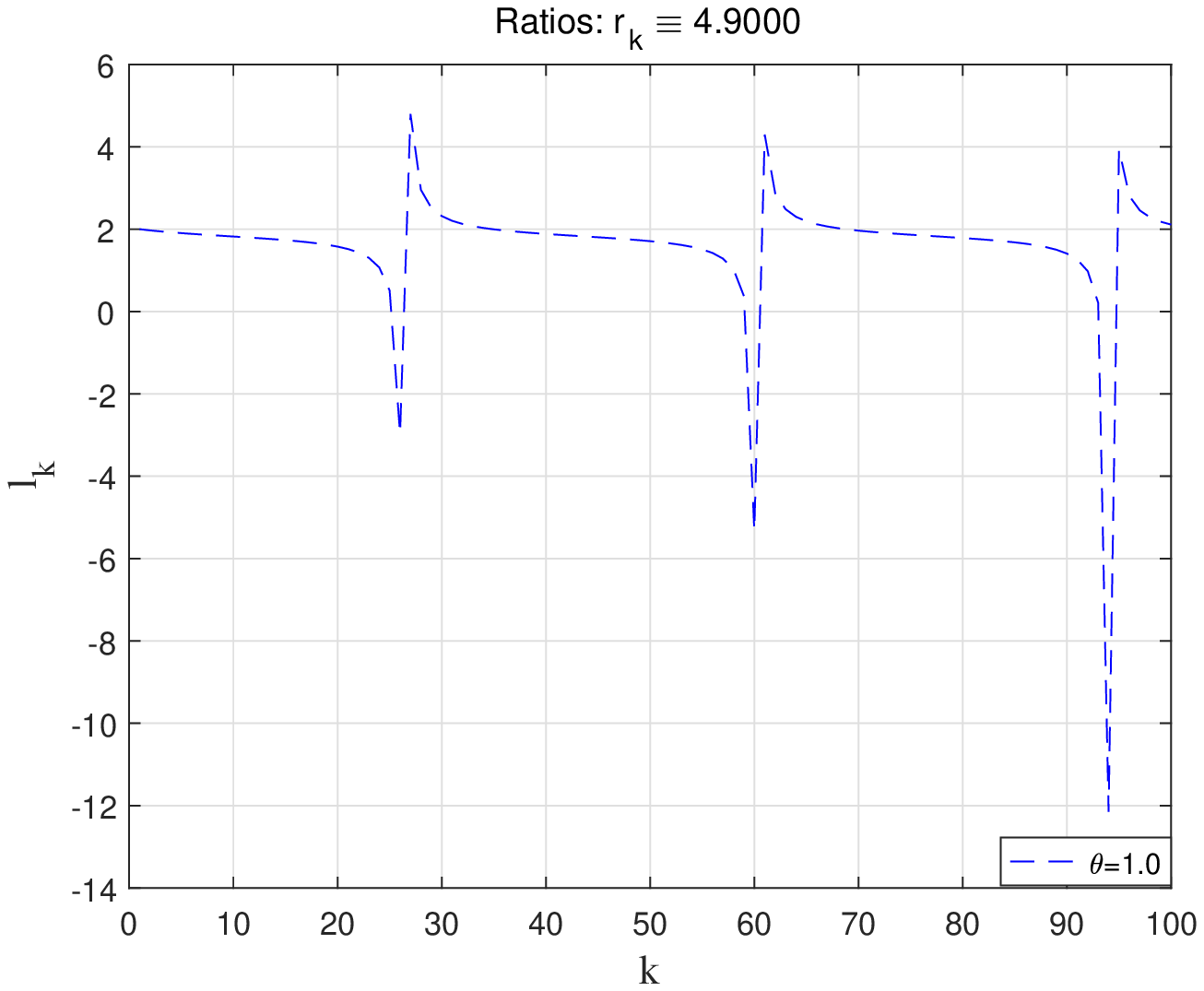}~~
    \includegraphics[width=6cm]{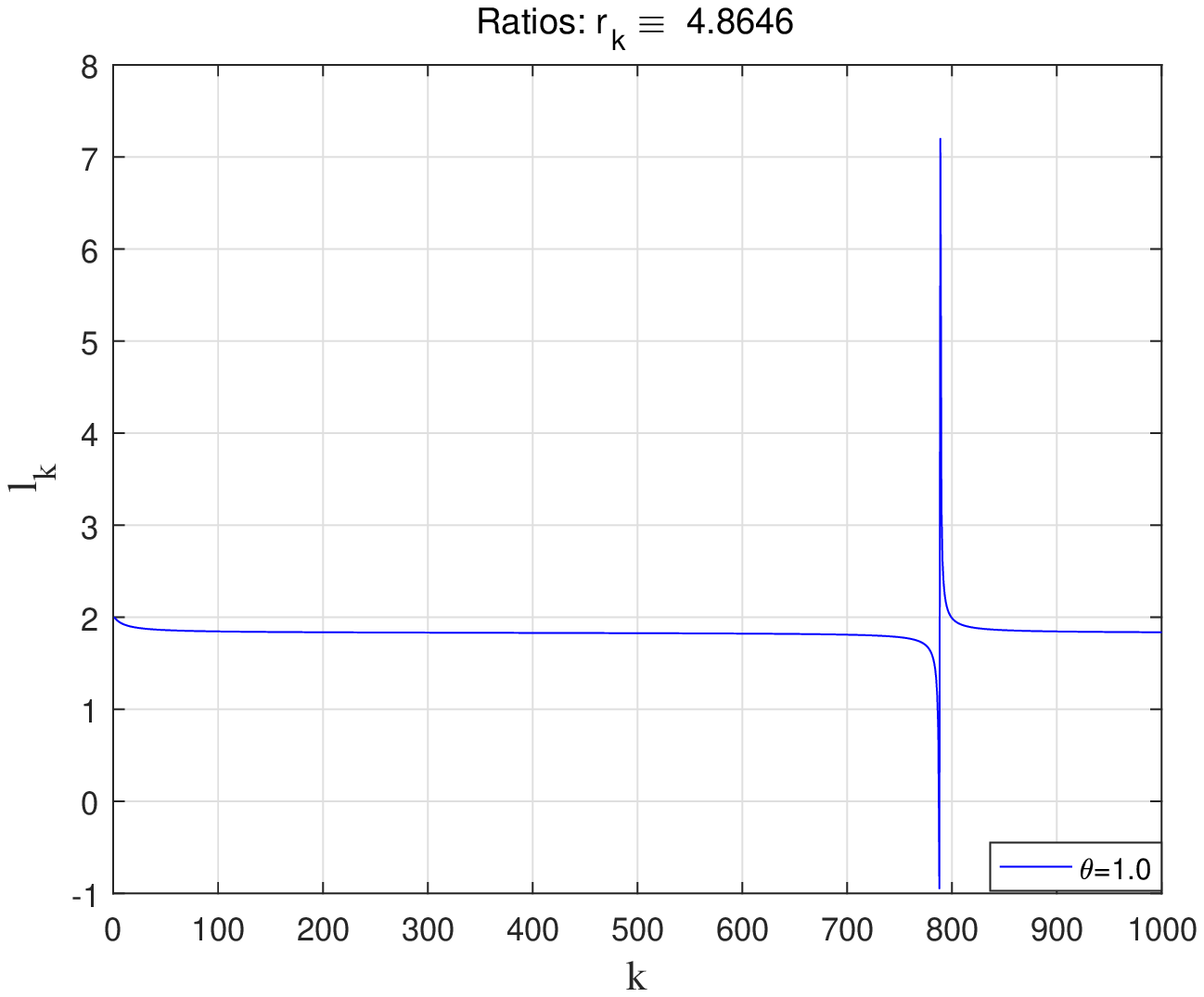}~~\\
    \includegraphics[width=6cm]{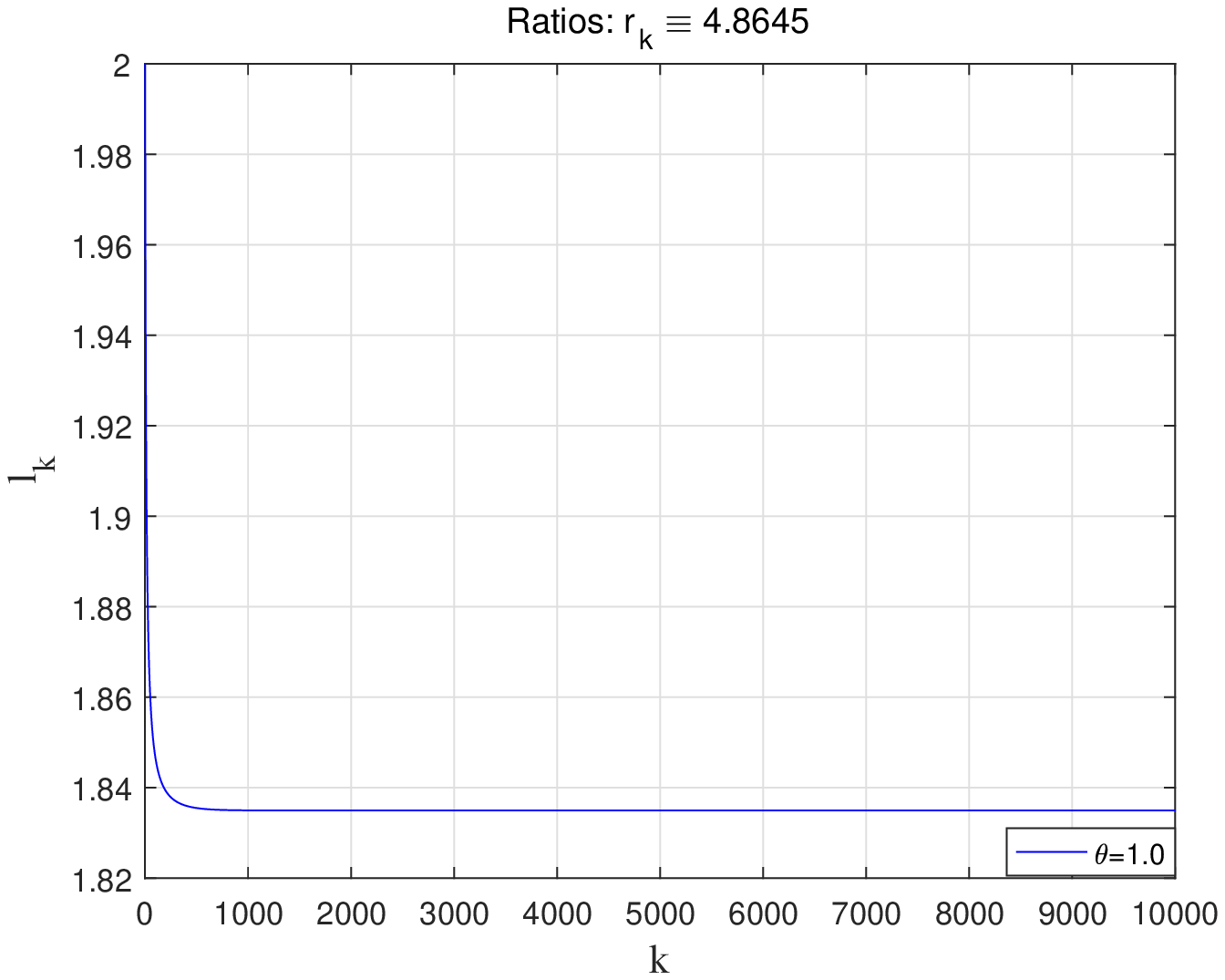}~~
    \includegraphics[width=6cm]{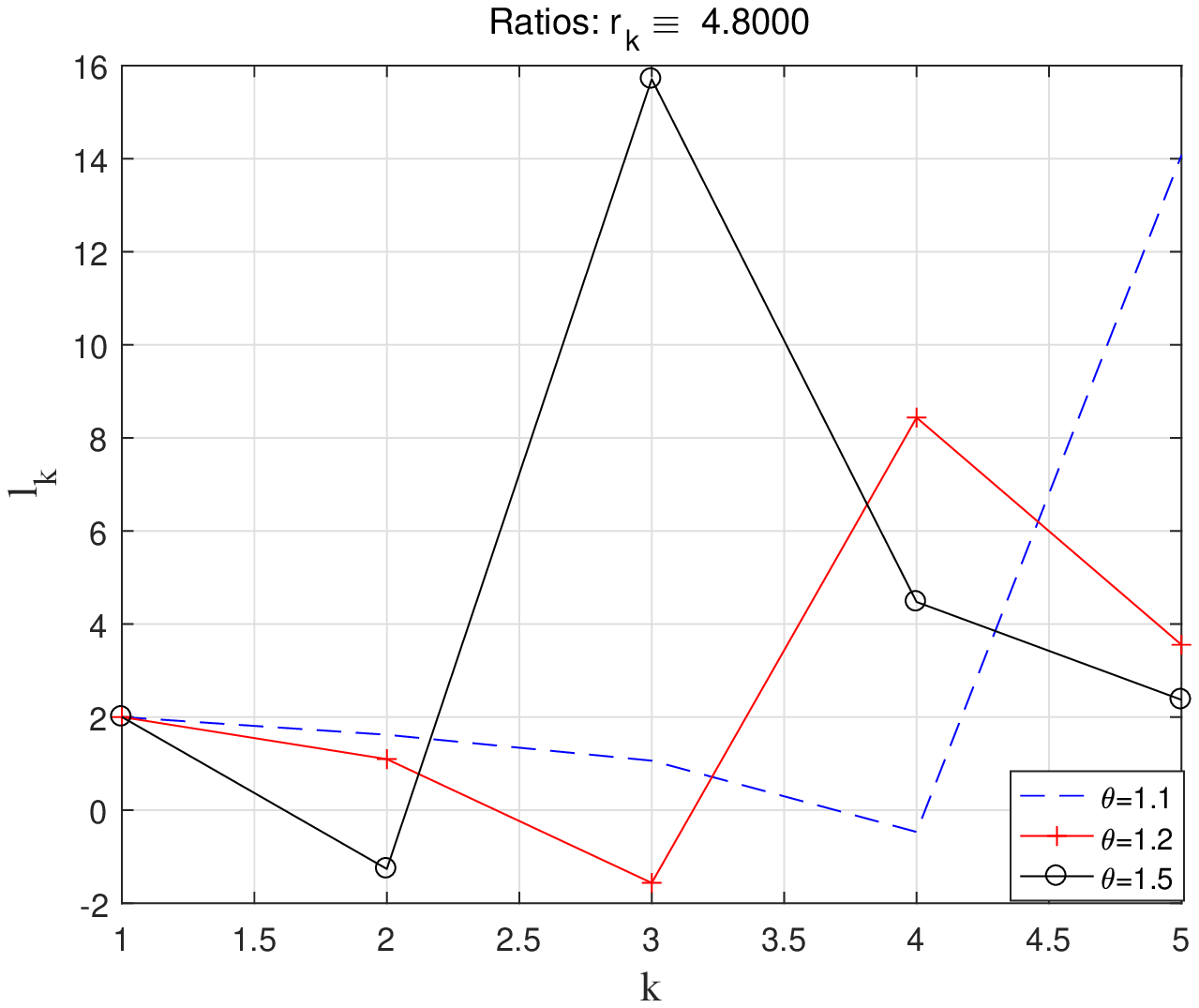}~~
    \caption{\small Simulations to  maximum (constant) ratios $r_s$ versus $\theta$.}  \label{Fig.2.1}
\end{figure}
%

\begin{lemma}\label{Le:posi}
If the WSBDF2 discrete convolution kernels $b_{n-k}^{(n)}$ in \eqref{1.5} are positive semi-definite, the DOC kernels $d_{n-k}^{(n)}$ defined in \eqref{1.6} are also positive semi-definite. For any real sequence $\{w_k\}_{k=1}^n$, it holds that
\begin{equation*}
\sum_{k=1}^nw_k\sum_{j=1}^kd_{k-j}^{(k)}w_j\geq0\quad{\rm for}~~n\geq1.
\end{equation*}
\end{lemma}
\begin{proof}
It is obvious that
\begin{equation*}
\sum^n_{k=1}w_k\sum^k_{j=1}d^{(k)}_{k-j}w_j=W^TDW ~~~~{\rm for}~~n \geq 1,
\end{equation*}
where $W=(w_1,w_2,\cdots,w_n)^T$ and the matrix $D$ is defined in \eqref{1.9}.

According to Lemma \ref{Le:upper} or  Lemma \ref{Le:WSBDF}, the matrix $B$ is positive semi-definite.
From \eqref{1.10}, it is easy to obtain the matrix $D$ is also positive semi-definite.
\end{proof}

\begin{corollary}\label{Cor:DOC}
The DOC kernels $d_{n-k}^{(n)}$  in \eqref{1.6}  fulfill
\begin{equation*}
\sum_{j=1}^nd_{n-j}^{(n)}\equiv\tau_n \quad{\rm such ~~that}\quad \sum_{k=1}^n\sum_{j=1}^kd_{k-j}^{(k)}\equiv t_n \quad{\rm for}~~n\geq1.
\end{equation*}
\end{corollary}
\begin{proof}
From the discrete convolution kernels in \eqref{1.5}, we have
\begin{equation*}
b_1^{(n)}\tau_{n-1}+b_0^{(n)}\tau_{n}\equiv1 \quad{\rm for}~~n\geq1.
\end{equation*}
Below we use the mathematical induction to prove the first equality
\begin{equation}\label{ad2.6}
\sum_{j=1}^kd_{k-j}^{(k)}\equiv\tau_k\quad{\rm for}~~k\geq1.
\end{equation}
For $k=1$,  we obtain
\begin{equation*}
\sum_{j=1}^1d_{1-j}^{(1)}=d_0^{(1)}=\frac{1}{b_0^{(1)}}\equiv\tau_1.
\end{equation*}
Suppose now \eqref{ad2.6} holds for all $2\leq k\leq n-1$, we need to prove that it holds for $k=n.$
\begin{equation*}
\sum_{j=1}^nd_{n-j}^{(n)}=\frac{1}{b_0^{(n)}}\left(1-b_1^{(n)}\sum_{j=1}^{n-1}d_{n-1-j}^{(n-1)}\right)=\frac{1}{b_0^{(n)}}\left(1-b_1^{(n)}\tau_{n-1}\right)\equiv\tau_n.
\end{equation*}
Summing \eqref{ad2.6} from $k=1$ to $n$, it is straightforward to obtain the second equality and complete the proof.
\end{proof}

\begin{lemma}\label{Le:2.3}
The DOC kernels $d_{n-k}^{(n)}$  in \eqref{1.6}  have an explicit formula
\begin{equation*}
d_{n-k}^{(n)}=\frac{1}{b_0^{(k)}}\prod_{i=k+1}^n\frac{(2\theta-1)r_i^2}{1+2\theta r_i} \quad{\rm for}~~1\leq k\leq n.
\end{equation*}
\end{lemma}
\begin{proof}
From the DOC kernels in \eqref{1.6}, it yields
\begin{equation*}
-b_{1}^{(k+1)}d_{n-k-1}^{(n)}=-\frac{b_{1}^{(k+1)}}{b_{0}^{(k+1)}}b_{0}^{(k+1)}d_{n-k-1}^{(n)}=-\frac{b_{1}^{(k+1)}}{b_{0}^{(k+1)}}\left(-b_{1}^{(k+2)}d_{n-k-2}^{(n)}\right)
=\prod_{i=k+1}^n\left(-\frac{b_{1}^{(i)}}{b_{0}^{(i)}}\right).
\end{equation*}
From the WSBDF2 convolution kernels in \eqref{1.5}, we obtain
\begin{equation*} d_{n-k}^{(n)}=-\frac{b_{1}^{(k+1)}}{b_{0}^{(k)}}d_{n-k-1}^{(n)}=\frac{1}{b_0^{(k)}}\prod_{i=k+1}^n\left(-\frac{b_{1}^{(i)}}{b_{0}^{(i)}}\right)=\frac{1}{b_0^{(k)}}\prod_{i=k+1}^n\frac{(2\theta-1)r_i^2}{1+2\theta r_i}\quad{\rm for}~~1\leq k\leq n.
\end{equation*}
\end{proof}

\section{Stability for WSBDF2 method}\label{Se:stab}
We first consider the energy stability of the WSBDF2 method \eqref{1.4} by defining a
(modified) discrete energy $E^k$,
\begin{equation}\label{3.1}
E^k:=\frac{(2\theta-1)r^{3/2}_{k+1}}{1+r_{k+1}}\tau_k\|\partial_\tau u^k\|^2+\|\nabla u^k\|^2\quad{\rm for}~~k\geq1,
\end{equation}
together with the initial energy $E^0:=\|\nabla u^0\|^2$.
\begin{theorem}\label{theoremad3.1}
Under the adjacent step ratios $r_k$ satisfy $0<r_k\leq r_s$ in \eqref{ad2.a4}, the discrete solution $u^n$ of the WSBDF2 time-stepping scheme \eqref{1.4} satisfies
\begin{equation}\label{3.10}
\partial_\tau E^k\leq2\left(\theta f^k+\left(1-\theta\right) f^{k-1},\partial_\tau u^k\right)\quad{\rm for}~~k\geq1,
\end{equation}
which holds the energy dissipation law. So the discrete solution is unconditionally stable in the energy norm,
\end{theorem}
\begin{equation*}
\sqrt{E^n}\leq\sqrt{E^0}+4C_\Omega\left(\|f^0\|+\sum_{k=1}^n\|\nabla_\tau f^k\|\right)\quad{\rm for}~~n\geq1.
\end{equation*}
\begin{proof}
Taking $\phi^k=2\nabla_\tau u^k$ in the weak form  \eqref{1.7} for $k\geq2$, we have
\begin{equation*}
2(D_2u^k,\nabla_\tau u^k)-2(\theta\Delta u^k+(1-\theta)\Delta u^{k-1},\nabla_\tau u^k)=2(\theta f^k+(1-\theta)f^{k-1},\nabla_\tau u^k).
\end{equation*}
According to Lemma \ref{Le:WSBDF}, it yields
\begin{equation*}
2(D_2u^k,\nabla_\tau u^k)\geq\frac{\left(2\theta-1\right)r^{3/2}_{k+1}}{1+r_{k+1}}\tau_k\|\partial_\tau u^k\|^2-\frac{\left(2\theta-1\right)r^{3/2}_{k}}{1+r_{k}}\tau_{k-1}\|\partial_\tau u^{k-1}\|^2.
\end{equation*}
With the help of the inequality $2a(a-b)\geq a^2-b^2$, we obtain
\begin{equation*}
\begin{split}
&-2\left(\theta\Delta u^k+\left(1-\theta\right)\Delta u^{k-1},\nabla_\tau u^k\right)=2\left(\theta\nabla u^k+\left(1-\theta\right)\nabla u^{k-1},\nabla u^k-\nabla u^{k-1}\right)\\
&=2\left(1-\theta\right)\left(\nabla u^k+\nabla u^{k-1},\nabla u^k-\nabla u^{k-1}\right)+2\left(2\theta-1\right)\left(\nabla u^k,\nabla u^k-\nabla u^{k-1}\right)\\
&\geq2\left(1-\theta\right)\left(\|\nabla u^k\|^2-\|\nabla u^{k-1}\|^2\right)+\left(2\theta-1\right)\left(\|\nabla u^k\|^2-\|\nabla u^{k-1}\|^2\right)\\
&=\|\nabla u^k\|^2-\|\nabla u^{k-1}\|^2.
\end{split}
\end{equation*}
It is easy to obtain that
\begin{equation}\label{3.2}
\nabla_{\tau}E^k\leq2\left(\theta f^k+\left(1-\theta\right) f^{k-1},\nabla_\tau u^k\right),\quad{\rm for}~~k\geq2.
\end{equation}
By taking $\phi^k=2\nabla_\tau u^1$ in the weak form \eqref{1.7} for the case $k=1$, we get
\begin{equation*}
2\tau_1\|\partial_\tau u^1\|^2+\|\nabla u^1\|^2-\|\nabla u^0\|^2\leq2\left(\theta f^1+\left(1-\theta\right) f^{0},\nabla_\tau u^1\right),
\end{equation*}
which implies
\begin{equation*}
\nabla_{\tau}E^1\leq2\left(\theta f^1+\left(1-\theta\right) f^{0},\nabla_\tau u^1\right).
\end{equation*}
Combining it with the general case \eqref{3.2}, one gets the discrete energy dissipation law \eqref{3.10}.
Summing the inequality \eqref{3.10} from $k=1$ to $n$, we have
\begin{equation}\label{3.4}
E^n\leq E^0+2 \sum_{k=1}^{n}\left(\theta f^k+\left(1-\theta\right) f^{k-1},\nabla_\tau u^k\right)\quad{\rm for}~~n\geq1.
\end{equation}
By applying the technique of time summation by parts \cite{T} and the Cauchy-Schwarz inequality, we obtain
\begin{equation*}
\begin{split}
&2\sum_{k=1}^{n}\left(\theta f^k+\left(1-\theta\right) f^{k-1},\nabla_\tau u^k\right)=2\theta \sum_{k=1}^{n}\left(f^k,\nabla_\tau u^k\right)+2\left(1-\theta\right) \sum_{k=1}^{n}\left(f^{k-1},\nabla_\tau u^k\right)\\
&\leq2\theta\left(\left(f^{n},u^n\right)-\sum_{k=2}^{n}\left(\nabla_\tau f^k,u^{k-1}\right)-\left(f^1,u^0\right)\right)\\
&\quad+2\left(1-\theta\right)\left(\left(f^{n-1},u^n\right)-\sum_{k=1}^{n-1}\left(\nabla_\tau f^k,u^k\right)-\left(f^0,u^0\right)\right)\\
&\leq2\theta C_\Omega\left(\sqrt{E^n}\|f^n\|+\sum_{k=2}^n \sqrt{E^{k-1}}\|\nabla_\tau f^k\|+\sqrt{E^0}\|f^1\|\right)\\
&\quad+2\left(1-\theta\right)C_\Omega\left(\sqrt{E^n}\|f^{n-1}\|+\sum_{k=1}^{n-1} \sqrt{E^{k}}\|\nabla_\tau f^k\|+\sqrt{E^0}\|f^0\|\right)\quad{\rm for}~~n\geq1,
\end{split}
\end{equation*}
where the Poincar{\'e} inequality has been used. It follows from \eqref{3.4} that
\begin{equation*}
\begin{split}
E^n&\leq E^0+2\theta C_\Omega\left(\sqrt{E^n}\|f^n\|+\sum_{k=2}^n \sqrt{E^{k-1}}\|\nabla_\tau f^k\|+\sqrt{E^0}\|f^1\|\right)\\
&\quad+2\left(1-\theta\right)C_\Omega\left(\sqrt{E^n}\|f^{n-1}\|+\sum_{k=1}^{n-1} \sqrt{E^{k}}\|\nabla_\tau f^k\|+\sqrt{E^0}\|f^0\|\right)\quad{\rm for}~~n\geq1.
\end{split}
\end{equation*}
Taking some integer $n_0\left(0\leq n_0\leq n\right)$ such that $E^{n_0}=\max_{0\leq j\leq n}E^{j}$. Taking $n:=n_0$ in the above inequality and apply the triangle inequality to obtain
\begin{equation*}
\begin{split}
E^{n_0}&\leq \sqrt{E^0}\sqrt{E^{n_0}}+2\theta C_\Omega \sqrt{E^{n_0}}\left(\|f^{n_0}\|+\sum_{k=2}^{n_0} \|\nabla_\tau f^k\|+\|f^1\|\right)\\
&\quad+2\left(1-\theta\right)C_\Omega \sqrt{E^{n_0}} \left(\|f^{n_0-1}\|+\sum_{k=1}^{n_0-1}\|\nabla_\tau f^k\|+\|f^0\|\right)\\
&\leq \sqrt{E^0}\sqrt{E^{n_0}}+4\theta C_\Omega \sqrt{E^{n_0}}\left(\|f^{1}\|+\sum_{k=2}^{n_0} \|\nabla_\tau f^k\|\right)\\
&\quad+4\left(1-\theta\right)C_\Omega \sqrt{E^{n_0}} \left(\|f^{0}\|+\sum_{k=1}^{n_0-1}\|\nabla_\tau f^k\|\right)\quad{\rm for}~~n\geq1,
\end{split}
\end{equation*}
where $f^{n_0}=f^{1}+\sum_{k=2}^{n_0}\nabla_\tau f^k$ and $f^{n_0-1}=f^{0}+\sum_{k=1}^{n_0-1}\nabla_\tau f^k$ have been used.
Hence, it yields
\begin{equation*}
\begin{split}
\sqrt{E^{n}}&\leq \sqrt{E^0}+4\theta C_\Omega\left(\|f^{1}\|+\sum_{k=2}^{n_0} \|\nabla_\tau f^k\|\right)+4\left(1-\theta\right)C_\Omega\left(\|f^{0}\|+\sum_{k=1}^{n_0-1}\|\nabla_\tau f^k\|\right)\\
&\leq \sqrt{E^0}+4\theta C_\Omega\left(\|f^0\|+\sum_{k=1}^{n} \|\nabla_\tau f^k\|\right)+4\left(1-\theta\right)C_\Omega\left(\|f^{0}\|+\sum_{k=1}^{n}\|\nabla_\tau f^k\|\right)\\
&=\sqrt{E^0}+4C_\Omega\left(\|f^0\|+\sum_{k=1}^{n} \|\nabla_\tau f^k\|\right)\quad{\rm for}~~n\geq1.
\end{split}
\end{equation*}
\end{proof}
If the exterior force $f(t)$ is zero-valued, the discrete energy dissipation law \eqref{3.10} gives
\begin{equation*}
E^k\leq E^{k-1}\quad{\rm for}~~k\geq1,
\end{equation*}
so that the variable-step WSBDF2 method \eqref{1.3} preserves the energy dissipation law at the discrete
levels. This property would be important in simulating the gradient flow problems.

\begin{remark}
From the point of energy technique \cite{ACYZ:20}, the Nevanlinna-Odeh multiplier is choosing $\mu_1=1$, namely, $\nabla_{\tau}u^{k}:=u^k-\mu_1u^{k-1}$
in Theorem \ref{theoremad3.1}. Moreover, Lemma \ref{Le:WSBDF} is similar to G-stability  for two-step multistep methods on variable grids.
\end{remark}

Now we establish the $L^2$ norm stability of the WSBDF2 method \eqref{1.4}.
\begin{theorem}\label{Th:stab}
If the WSBDF2 kernels $b_{n-k}^{(n)}$ in \eqref{1.5} are positive semi-definite, the discrete solution $u^n$ of the WSBDF2 method \eqref{1.3} is unconditionally stable in the $L^2$ norm
\begin{equation*}
\|u^{n}\|\leq\|u^0\|+2\sum_{k=1}^{n}\sum_{j=1}^{k}d_{k-j}^{(k)}\|\theta f^j+(1-\theta)f^{j-1}\|
\leq\|u^0\|+2t_n\max_{1\leq j\leq n}\|f^j\|\quad{\rm for} ~~n \geq 1.
\end{equation*}
Thus, the WSBDF2 method \eqref{1.3} is monotonicity-preserving.
\end{theorem}
\begin{proof}
Multiplying both sides of the equation \eqref{1.3} by the DOC kernels $d_{k-j}^{(k)}$, and summing $j$ from $1$ to $k$, we derive
\begin{equation*}
\sum_{j=1}^kd_{k-j}^{(k)}D_2u^j-\sum_{j=1}^kd_{k-j}^{(k)}\left(\theta\Delta u^j+\left(1-\theta\right)\Delta u^{j-1}\right)=\sum_{j=1}^kd_{k-j}^{(k)}\left(\theta f^j+\left(1-\theta\right)f^{j-1}\right).
\end{equation*}
Applying the orthogonal identity \eqref{1.8}, it yields
 \begin{equation*}
\sum_{j=1}^kd_{k-j}^{(k)}D_2u^j=\nabla_\tau u^k\quad{\rm for}~~k \geq 1.
\end{equation*}
Hence, we have
\begin{equation}\label{3.5}
\nabla_\tau u^k-\sum_{j=1}^kd_{k-j}^{(k)}\left(\theta\Delta u^j+\left(1-\theta\right)\Delta u^{j-1}\right)=\sum_{j=1}^kd_{k-j}^{(k)}\left(\theta f^j+\left(1-\theta\right)f^{j-1}\right).
\end{equation}
Making the inner product of the equation \eqref{3.5} with $\psi^k=2\left(\theta u^k+\left(1-\theta\right) u^{k-1}\right)$, and summing the resulting equality from $k=1$ to $n$, there exists
\begin{equation*}
\begin{split}
&\sum_{k=1}^n\left(\nabla_\tau u^k,\psi^k\right)-\sum_{k=1}^{n}\sum_{j=1}^kd_{k-j}^{(k)}\left(\theta\Delta u^j+\left(1-\theta\right)\Delta u^{j-1},\psi^k\right)\\
&=\sum_{k=1}^{n}\sum_{j=1}^kd_{k-j}^{(k)}\left(\theta f^j+\left(1-\theta\right)f^{j-1},\psi^k\right)\quad{\rm for}~~n \geq 1.
\end{split}
\end{equation*}
For the first term on the left hand, we have
\begin{equation}\label{3.6}
\begin{split}
&\sum_{k=1}^n\left(\nabla_\tau u^k,\psi^k\right)=2\sum_{k=1}^n\left(u^k-u^{k-1},\theta u^k+\left(1-\theta\right) u^{k-1}\right)\\
&=\sum_{k=1}^n\left[2\left(1-\theta\right)\left(u^k-u^{k-1}, u^k+ u^{k-1}\right)+2\left(2\theta-1\right)\left(u^k-u^{k-1}, u^k\right)\right]\\
&\geq\sum_{k=1}^n\left[2\left(1-\theta\right)\left(\|u^k\|^2-\|u^{k-1}\|^2\right)
+\left(2\theta-1\right)\left(\|u^k\|^2-\|u^{k-1}\|^2\right)\right]\\
&=\sum_{k=1}^n\left(\|u^k\|^2-\|u^{k-1}\|^2\right)=\|u^n\|^2-\|u^0\|^2,
\end{split}
\end{equation}
where the inequality $2a(a-b)\geq a^2-b^2$ has been used.

For the second term on the left hand, we obtain
\begin{equation}\label{3.7}
\begin{split}
&-\sum_{k=1}^{n}\sum_{j=1}^kd_{k-j}^{(k)}\left(\theta\Delta u^j+\left(1-\theta\right)\Delta u^{j-1},\psi^k\right)\\
&=-2\sum_{k=1}^{n}\sum_{j=1}^kd_{k-j}^{(k)}\left(\theta\Delta u^j+\left(1-\theta\right)\Delta u^{j-1},\theta u^k+\left(1-\theta\right) u^{k-1}\right)\\
&=2\sum_{k=1}^{n}\sum_{j=1}^kd_{k-j}^{(k)}\left(\theta\nabla u^j+\left(1-\theta\right)\nabla u^{j-1},\theta \nabla u^k+\left(1-\theta\right)\nabla u^{k-1}\right)\geq0,
\end{split}
\end{equation}
where the Lemma \ref{Le:posi} has been used.

Combing \eqref{3.6}, \eqref{3.7} and the Cauchy-Schwarz inequality, it yields
\begin{equation*}
\|u^n\|^2\leq\|u^0\|^2+2\sum_{k=1}^{n}\|\theta u^k+\left(1-\theta\right) u^{k-1}\|\sum_{j=1}^{k}d_{k-j}^{(k)}\|\theta f^j+\left(1-\theta\right)f^{j-1}\|.
\end{equation*}
Taking some integer $n_1\left(0\leq n_1\leq n\right)$ such that $\|u^{n_1}\|=\max_{0\leq k\leq n}\|u^k\|$. Taking $n:=n_1$ in the above inequality, we get
\begin{equation}
\|u^{n_1}\|^2\leq\|u^0\|\|u^{n_1}\|+2\|u^{n_1}\|\sum_{k=1}^{n_1}\sum_{j=1}^{k}d_{k-j}^{(k)}\|\theta f^j+\left(1-\theta\right)f^{j-1}\|\quad{\rm for} ~~n \geq 1.
\end{equation}
Hence, it yields
\begin{equation*}
\begin{split}
\|u^{n}\|\leq\|u^0\|+2\sum_{k=1}^{n}\sum_{j=1}^{k}d_{k-j}^{(k)}\|\theta f^j+(1-\theta)f^{j-1}\|
\leq\|u^0\|+2t_n\max_{1\leq j\leq n}\|f^j\|\quad{\rm for} ~~n \geq 1.
\end{split}
\end{equation*}
where the Corollary \ref{Cor:DOC} has been used.
\end{proof}

\section{$L^2$ norm convergence for WSBDF2 method}\label{Se:conv}

\begin{lemma}\label{Le:cons}
For the consistency error $\eta^j:=D_2u(t_j)-\theta u'(t_j)-\left(1-\theta\right) u'(t_{j-1})$, it holds that
\begin{equation*}
\sum_{k=1}^{n}\sum_{j=1}^kd_{k-j}^{(k)}\|\eta^j\|\leq2\left(2-\theta\right)\tau_{\max}^2\max_{0<t\leq t_1}\|\partial_{tt}u\|+3t_n\tau_{\max}^2\max_{0<t\leq T}\|\partial_{ttt}u\|\quad \forall n \geq 1
\end{equation*}
with $\tau_{\max}=\max_{1\leq k \leq N}{\tau_k}$.
\end{lemma}
\begin{proof}
For simplicity, denote
\begin{equation*}
G_{t2}^j=\int_{t_{j-1}}^{t_j}\|\partial_{tt}u\|dt \quad{\rm  and}\quad
G_{t3}^j=\int_{t_{j-1}}^{t_j}\|\partial_{ttt}u\|dt \quad{\rm  for}~~~ j\geq 1.
\end{equation*}
For the case of $j=1$, by using the Taylor's expansion formula, we obtain
\begin{equation*}
\begin{split}
\eta^1&=\frac{u(t_1)-u(t_0)}{\tau_1}-\theta u'(t_1)-\left(1-\theta\right) u'(t_{0})\\
&=\left(1-\theta\right)\int_0^{t_1}\partial_{tt}udt-\frac{1}{\tau_1}\int_0^{t_1}t\partial_{tt}udt.
\end{split}
\end{equation*}
Then the consistency error is bounded by
\begin{equation*}
\begin{split}
\|\eta^1\|\leq\left(1-\theta\right)G_{t2}^1+G_{t2}^1=\left(2-\theta\right)G_{t2}^1.
\end{split}
\end{equation*}
For the case of $j\geq2$, by using the Taylor's expansion formula, we derive
\begin{equation*}
\begin{split}
\eta^j&=b_0^{(j)}\left(u(t_j)-u(t_{j-1})\right)+b_1^{(j)}\left(u(t_{j-1})-u(t_{j-2})\right)-\theta u'(t_j)-\left(1-\theta\right) u'(t_{j-1})\\
&=\frac{\left(1+2\theta r_j\right)-\left(1-2\theta\right)r_j^2}{2\tau_j\left(1+r_j\right)}\int_{t_{j-1}}^{t_j}\left(t-t_{j-1}\right)^2\partial_{ttt}udt\\
&\quad+\frac{\left(1-2\theta\right)r_j^2}{2\tau_j\left(1+r_j\right)}\int_{t_{j-2}}^{t_j}\left(t-t_{j-2}\right)^2\partial_{ttt}udt
+\left(1-\theta\right)\int_{t_{j-1}}^{t_j}\left(t_{j-1}-t\right)\partial_{ttt}udt\\
&=\frac{1}{2}\left(b_0^{(j)}-b_1^{(j)}\right)\int_{t_{j-1}}^{t_j}\left(t-t_{j-1}\right)^2\partial_{ttt}udt
+\frac{1}{2}b_1^{(j)}\int_{t_{j-2}}^{t_{j-1}}\left(t-t_{j-2}\right)^2\partial_{ttt}udt\\
&\quad+\frac{1}{2}b_1^{(j)}\int_{t_{j-1}}^{t_{j}}\left(t-t_{j-1}+\tau_{j-1}\right)^2\partial_{ttt}udt
+\left(1-\theta\right)\int_{t_{j-1}}^{t_j}\left(t_{j-1}-t\right)\partial_{ttt}udt\\
&=\frac{1}{2}b_0^{(j)}\int_{t_{j-1}}^{t_j}\left(t-t_{j-1}\right)^2\partial_{ttt}udt
+\frac{1}{2}b_1^{(j)}\int_{t_{j-2}}^{t_{j-1}}\left(t-t_{j-2}\right)^2\partial_{ttt}udt\\
&\quad+\frac{1}{2}b_1^{(j)}\tau_{j-1}\int_{t_{j-1}}^{t_{j}}\left(2\left(t-t_{j-1}\right)+\tau_{j-1}\right)\partial_{ttt}udt
+\left(1-\theta\right)\int_{t_{j-1}}^{t_j}\left(t_{j-1}-t\right)\partial_{ttt}udt.
\end{split}
\end{equation*}
Then the consistency error is bounded by
\begin{equation*}
\begin{split}
\|\eta^j\|&\leq\frac{1}{2}b_0^{(j)}\tau_j^2G_{t3}^j-\frac{1}{2}b_1^{(j)}\tau_{j-1}^2G_{t3}^{j-1}
-\frac{1}{2}b_1^{(j)}\tau_{j-1}\left(2\tau_j+\tau_{j-1}\right)G_{t3}^j
+\left(1-\theta\right)\tau_jG_{t3}^j\\
&=\frac{1}{2}b_0^{(j)}\tau_j^2G_{t3}^j\left(1-\frac{b_1^{(j)}\left(1+2r_j\right)}{b_0^{(j)}r_j^2}+\frac{2\left(1-\theta\right)}{b_0^{(j)}\tau_j}\right)
-\frac{b_1^{(j)}\tau_{j-1}^2}{2b_0^{(j)}}b_0^{(j)}G_{t3}^{j-1}\\
&=b_0^{(j)}\tau_j^2G_{t3}^j+\frac{\left(2\theta-1\right)r_j^2\tau_{j-1}^2}{2\left(1+2\theta r_j\right)}b_0^{(j)}G_{t3}^{j-1}\quad {\rm for} ~~~~ j \geq 2.
\end{split}
\end{equation*}
Recalling the definitions of WSBDF2 kernels \eqref{1.5} and DOC kernels \eqref{1.6}, it yields
\begin{equation*}
d_{k-j}^{(k)}b_0^{(j)}=-d_{k-j-1}^{(k)}b_1^{(j+1)}=\frac{\left(2\theta-1\right)r_{j+1}^2}{1+2\theta r_{j+1}}d_{k-j-1}^{(k)}b_0^{(j+1)}\quad {\rm for} ~~~~ 1\leq j \leq k-1.
\end{equation*}
Thus, we apply the triangle inequality, we get
\begin{equation*}
\begin{split}
&\sum_{j=1}^kd_{k-j}^{(k)}\|\eta^j\|=d_{k-1}^{(k)}\|\eta^1\|+\sum_{j=2}^{k}d_{k-j}^{(k)}\|\eta^j\|\\
&\leq d_{k-1}^{(k)}\left(2-\theta\right)G_{t2}^1+\sum_{j=2}^{k}d_{k-j}^{(k)}b_0^{(j)}\tau_j^2G_{t3}^j+\frac{1}{2}\sum_{j=1}^{k-1}d_{k-j-1}^{(k)}b_0^{(j+1)}\frac{\left(2\theta-1\right)r_{j+1}^2}{1+2\theta r_{j+1}}\tau_{j}^2G_{t3}^{j}\\
&= d_{k-1}^{(k)}\left(2-\theta\right)G_{t2}^1+\sum_{j=2}^{k}d_{k-j}^{(k)}b_0^{(j)}\tau_j^2G_{t3}^j+\frac{1}{2}\sum_{j=1}^{k-1}d_{k-j}^{(k)}b_0^{(j)}\tau_{j}^2G_{t3}^{j}\\
&\leq d_{k-1}^{(k)}\left(2-\theta\right)G_{t2}^1+\frac{3}{2}\sum_{j=1}^{k}d_{k-j}^{(k)}b_0^{(j)}\tau_{j}^2G_{t3}^{j}\quad {\rm for} ~~~~ k \geq 1.
\end{split}
\end{equation*}
From Lemma \ref{Le:2.3}, we get
\begin{equation*}
d_{k-1}^{(k)}=\frac{1}{b_0^{(1)}}\prod_{i=2}^k\frac{(2\theta-1)r_i^2}{1+2\theta r_i}\leq\tau_1\prod_{i=2}^k\frac{r_i^2}{1+2 r_i}=\tau_k\prod_{i=2}^k\frac{r_i}{1+2 r_i}\leq\frac{\tau_k}{2^{k-1}}\quad {\rm for} ~~~~ 1\leq k \leq n.
\end{equation*}
Thus, we have
\begin{equation*}
\sum_{k=1}^nd_{k-1}^{(k)}\leq\sum_{k=1}^n\frac{\tau_k}{2^{k-1}}\leq\tau_{\max}\sum_{k=1}^n\frac{1}{2^{k-1}}\leq2\tau_{\max}~~{\rm with}~~\tau_{\max}=\max_{1\leq k \leq N}{\tau_k}.
\end{equation*}
According to the triangle inequality and Corollary \ref{Cor:DOC}, we derive
\begin{equation*}
\begin{split}
\sum_{k=1}^{n}\sum_{j=1}^kd_{k-j}^{(k)}\|\eta^j\|&\leq\sum_{k=1}^{n}d_{k-1}^{(k)}\left(2-\theta\right)G_{t2}^1+3\sum_{k=1}^{n}\sum_{j=1}^{k}d_{k-j}^{(k)}\tau_{j}G_{t3}^{j}\\
&\leq2\left(2-\theta\right)\tau_{\max}^2\max_{0<t\leq t_1}\|\partial_{tt}u\|+3t_n\tau_{\max}^2\max_{0<t\leq T}\|\partial_{ttt}u\|.
\end{split}
\end{equation*}

\end{proof}

\begin{theorem}\label{Th:conv}
Let $u(t_n)$ and $u^n$ be the solutions of the parabolic equation \eqref{ivp} and the WSBDF2 method \eqref{1.4}, respectively. Then the time-discrete solution $u^n$ is convergent in the $L^2$ norm
\begin{equation*}
\|u(t_n)-u^n\|\leq C\left(\tau_{\max}^2\max_{0<t\leq t_1}\|\partial_{tt}u\|+t_n\tau_{\max}^2\max_{0<t\leq T}\|\partial_{ttt}u\|\right)~~{\rm for}~~n \geq 1.
\end{equation*}
\end{theorem}

\begin{proof}
Let $e^n=u(t_n)-u^n$. From \eqref{ivp} and \eqref{1.4}, we obtain
\begin{equation}\label{4.3}
D_2e^n-\theta\Delta e^n-\left(1-\theta\right)\Delta e^{n-1}=D_2u(t_n)-\theta u'(t_n)-\left(1-\theta\right) u'(t_{n-1}):=\eta^n~~~~{\rm for} ~~n \geq 1.
\end{equation}
Replacing $n$ by $j$ and multiplying both sides of the equation \eqref{4.3} by the DOC kernels $d_{k-j}^{(k)}$, and summing $j$ from $1$ to $k$, we obtain
\begin{equation*}
\sum_{j=1}^kd_{k-j}^{(k)}D_2e^j-\sum_{j=1}^kd_{k-j}^{(k)}\left(\theta\Delta e^j+\left(1-\theta\right)\Delta e^{j-1}\right)=\sum_{j=1}^kd_{k-j}^{(k)}\eta^j~~~~{\rm for} ~~~~ k \geq 1.
\end{equation*}
Applying the orthogonal identity \eqref{1.8}, it yields
\begin{equation*}
\sum_{j=1}^kd_{k-j}^{(k)}D_2e^j=\nabla_{\tau}e^k~~~~{\rm for} ~~~~ k \geq 1.
\end{equation*}
Hence, we have
\begin{equation}\label{4.4}
\nabla_{\tau}e^k-\sum_{j=1}^kd_{k-j}^{(k)}\left(\theta\Delta e^j+\left(1-\theta\right)\Delta e^{j-1}\right)=\sum_{j=1}^kd_{k-j}^{(k)}\eta^j   ~~~~{\rm for} ~~~~ k \geq 1.
\end{equation}
Making the inner product of the equation \eqref{4.4} with $\phi^k=2\left(\theta e^k+\left(1-\theta\right)e^{k-1}\right)$, and summing the resulting equality
from $k=1$ to $n$, there exists
\begin{equation*}
\sum_{k=1}^n\left(\nabla_{\tau}e^k,\phi^k\right)-\sum_{k=1}^n\sum_{j=1}^kd_{k-j}^{(k)}\left(\theta\Delta e^j+\left(1-\theta\right)\Delta e^{j-1},\phi^k\right)=\sum_{k=1}^n\sum_{j=1}^kd_{k-j}^{(k)}\left(\eta^j,\phi^k\right).
\end{equation*}
For the first term on the left hand, we have
\begin{equation}\label{4.5}
\begin{split}
\sum_{k=1}^n\left(\nabla_{\tau}e^k,\phi^k\right)&=2\sum_{k=1}^n\left(e^k-e^{k-1},\theta e^k+\left(1-\theta\right)e^{k-1}\right)\\
&=\sum_{k=1}^n\left[2\left(1-\theta\right)\left(e^k-e^{k-1},e^k+e^{k-1}\right)+2\left(2\theta-1\right)\left(e^k-e^{k-1},e^k\right)\right]\\
&\geq\sum_{k=1}^n\left[2\left(1-\theta\right)\left(\|e^k\|^2-\|e^{k-1}\|^2\right)+\left(2\theta-1\right)\left(\|e^k\|^2-\|e^{k-1}\|^2\right)\right]\\
&=\sum_{k=1}^n\left(\|e^k\|^2-\|e^{k-1}\|^2\right)=\|e^n\|^2-\|e^0\|^2,
\end{split}
\end{equation}
where the inequality $2a(a-b)\geq a^2-b^2$ has been used.

For the second term on the left hand, we obtain
\begin{equation}\label{4.6}
\begin{split}
&-\sum_{k=1}^n\sum_{j=1}^kd_{k-j}^{(k)}\left(\theta\Delta e^j+\left(1-\theta\right)\Delta e^{j-1},\phi^k\right)\\
&=-2\sum_{k=1}^n\sum_{j=1}^kd_{k-j}^{(k)}\left(\theta\Delta e^j+\left(1-\theta\right)\Delta e^{j-1},\theta e^k+\left(1-\theta\right)e^{k-1}\right)\\
&=2\sum_{k=1}^n\sum_{j=1}^kd_{k-j}^{(k)}\left(\theta \nabla e^j+\left(1-\theta\right)\nabla e^{j-1},\theta \nabla e^k+\left(1-\theta\right)\nabla e^{k-1}\right)\geq0,
\end{split}
\end{equation}
where the Lemma \ref{Le:posi} has been used.

Combing \eqref{4.5}, \eqref{4.6} and the Cauchy-Schwarz inequality, it yields
\begin{equation*}
\|e^n\|^2\leq2\sum_{k=1}^n\|\theta e^k+\left(1-\theta\right)e^{k-1}\|\sum_{j=1}^kd_{k-j}^{(k)}\|\eta^j\|~~~~{\rm for} ~~~~ n \geq 1.
\end{equation*}
Taking some integer $n_1\left(0\leq n_1\leq n\right)$  such that $\|e^{n_1}\|=\max_{0\leq k\leq n}\|e^k\|$. Taking $n:=n_1$ in the above inequality, we get
\begin{equation*}
\|e^{n_1}\|^2\leq 2\|e^{n_1}\|\sum_{k=1}^{n_1}\sum_{j=1}^kd_{k-j}^{(k)}\|\eta^j\|~~~~{\rm for} ~~~~ n \geq 1.
\end{equation*}
Hence, it yields
\begin{equation*}
\|e^{n}\|\leq\|e^{n_1}\|\leq 2\sum_{k=1}^{n_1}\sum_{j=1}^kd_{k-j}^{(k)}\|\eta^j\|\leq 2\sum_{k=1}^{n}\sum_{j=1}^kd_{k-j}^{(k)}\|\eta^j\|~~~~{\rm for} ~~~~ n \geq 1.
\end{equation*}
The desired results follows from Lemma \ref{Le:cons} immediately.
\end{proof}

\section{Numerical experiments}\label{Se:numer}
We apply the WSBDF2 method \eqref{1.3} to the initial and boundary value problems \eqref{ivp}
with  $\Omega=(-1,1)^2$ and $T=1$, subject to homogeneous Dirichlet boundary conditions.
In space, we discretize by the spectral collocation method  with the Cheby\-shev--Gauss--Lobatto points.
We numerically verified the  theoretical results including convergence orders in the discrete $L^2$-norm.
We express the space discrete approximation $u_I^{n}$ in terms of its values
at Chebyshev--Gauss--Lobatto  points,
\[u_I^{n}(x,y)=\sum^{M_x}_{i=0}\sum^{M_y}_{j=0}u_{ij}^{n}\ell_i(x)\ell_j(y),
\quad \ell_i(x)=\prod_{\substack{j=0\\ j\ne i}}^{M_x}\frac{x-x_j}{x_i-x_j},\]
where  $u_{ij}^{n}:=u_I^{n}(x_i,y_j)$ at the mesh points  $(x_i,y_j)$.
Here, $-1=x_0<x_1<\dotsb<x_{M_x}=1$ and $-1=y_0<y_1<\dotsb<y_{M_y}=1$
are nodes of Lobatto quadrature rules.
In order to test the temporal error, we fix $M_x = M_y = 20$;
the spatial error is negligible  since the spectral collocation method converges exponentially;
see, e.g., \cite[Theorem 4.4, \textsection{4.5.2}]{STW:2011}.

\begin{example}\label{ex:6.1}
The initial value and the forcing term are chosen such that the exact solution of equation \eqref{ivp} is
\begin{equation}
\label{periodic}
u(x,y,t)=(t^3+1)\sin(\pi x)\sin(\pi y),\quad -1\leqslant x,y \leqslant 1,\  0\leqslant t \leqslant 1.
\end{equation}
Here, we consider three cases of the adjacent time-step ratios $r_k$.

Case I: $r_{2k}=4,$ for $1\leq k\leq \frac{N}{2}$, and $r_{2k+1}=1/4,$ for $1\leq k\leq \frac{N}{2}-1$.

Case II: $r_{k}=2,$ for $1\leq k\leq N$.

Case III: the arbitrary meshes with random time-steps $\tau_k=T\epsilon_k/S$ for $1\leq k\leq N$, where
$S=\sum_{k=1}^N\epsilon_k$ and $\epsilon_k\in(0,1)$ are random numbers subject to the uniform distribution \cite{LZ}.

\end{example}

\begin{table}[!ht]
 \begin{center}
  \caption { The discrete $L^2$-norm errors and numerical convergence orders with $M_x=M_y=20$.}
\begin{tabular*}{\linewidth}{@{\extracolsep{\fill}}*{10}{c}}                   \hline
&&&Case I\\
$N$& $\theta=1/2$&  Rate     & $\theta=3/4$  & Rate        & $\theta=1$ &   Rate    \\\hline
~20&    1.9072e-04  &           &  2.5972e-04     &          & 3.3563e-04    &  \\
~40&    4.4006e-05  &  2.1157  &  6.2088e-05     &  2.0646     & 8.0966e-05    &2.0515  \\
~80&    1.0542e-05   & 2.0615    &1.5139e-05     & 2.0360     & 1.9832e-05    & 2.0295  \\
~160&   2.5781e-06  & 2.0318     &  3.7351e-06    & 2.0191     & 4.9038e-06    & 2.0158  \\ \hline
&&&Case II\\
$N$& $\theta=1/2$&  Rate       & $\theta=3/4$  & Rate       & $\theta=1$ &   Rate    \\\hline
~20&   1.1216e-02  &       & 1.7255e-02     &      &  2.1030e-02    &  \\
~40&   1.1216e-02  &  -     & 1.7255e-02     & -     &  2.1030e-02    & - \\
~80&   1.1216e-02  &  -     & 1.7255e-02     & -     &  2.1030e-02    & -  \\
~160&  1.1216e-02  &  -     & 1.7255e-02     & -     &  2.1030e-02    &  -  \\ \hline
&&&Case III\\
$N$& $\theta=1/2$&  Rate       & $\theta=3/4$  & Rate       & $\theta=1$ &   Rate    \\\hline
~20&   1.4926e-04  &            &  2.6837e-04     &          &  3.9991e-04    &  \\
~40&   3.0490e-05  & 2.2914      &  5.5907e-05     &2.2631   &8.3039e-05    & 2.2678 \\
~80&   6.7401e-06  &  2.1775     & 1.5243e-05     &1.8749     &2.3421e-05    & 1.8260  \\
~160&  1.6679e-06  &  2.0147      & 3.0404e-06     &2.3258    &4.2701e-06    & 2.4555  \\ \hline
    \end{tabular*}\label{table:1}
  \end{center}
\end{table}

Table \ref{table:1} shows  the optimal convergence orders,
which agree with Theorem \ref{Th:conv}.

\section{Conclusions}
There are already some theoretical results for BDF2 method or Crank-Nicolson method to solve the PDEs with variable steps.
In this work,  we first construct  the WSBDF2 methods by weight and shifted technique,  which  establish a connection between BDF2 and Crank-Nicolson scheme.
Here, the optimal adjacent  time-step  ratios are obtained for the WSBDF2 methods,  which greatly amplify the maximum time-step ratios for BDF2 in recent works  \cite{LZ,LJWZ,TZZ:20}.
The main results  of this paper is to prove the unconditional stability and optimal convergence, which fill in the gap  of the    theory for PDEs in \cite{WR:08}.
Another interesting topic is how to design the high-order schemes (e.g., WSBDF3) on variable grids.

\section*{Acknowledgments}
The first author wishes to thank Prof.  Hong-lin Liao, Jiwei Zhang and Zhimin Zhang  for several suggestions and comments.
The work was partially supported by NSFC 11601206.

\bibliographystyle{amsplain}

\end{document}